\documentclass[12pt]{article}
\usepackage{amssymb}
\usepackage{amsthm}
\newtheorem{Definition}{Definition}
\newtheorem{Lemma}[Definition]{Lemma}

\newtheorem{Theorem}[Definition]{Theorem}
\newtheorem{Remark}[Definition]{Remark}
\newtheorem{Example}[Definition]{Example}

\title{General coupled semirings of residuated lattices\thanks{Preprint of an article published by Elsevier in {\it Fuzzy Sets and Systems} {\bf303} (2016), 128-135. It is available online at: {\texttt https://www.sciencedirect.com/science/article/pii/S0165011415005898}.}}
\author{Ivan~Chajda and Helmut~L\"anger}
\date{}
\begin{document}
\footnotetext[1]{Support of the research of both authors by the bilateral project "New perspectives on residuated posets", supported by the Austrian Science Fund (FWF), project I~1923-N25, and the Czech Science Foundation (GA\v CR), project 15-34697L, as well as by the project "Ordered structures for algebraic logic", supported by AKTION Austria -- Czech Republic, project 71p3, is gratefully acknowledged.}
\footnotetext[2]{helmut.laenger@tuwien.ac.at, telephone: +4315880110412, fax: +4315880110499}
\maketitle
\begin{abstract}
Di~Nola and Gerla showed that MV-algebras and coupled semirings are in a natural one-to-one correspondence. We generalize this correspondence to residuated lattices satisfying the double negation law.
\end{abstract} 



It was shown by Di~Nola and Gerla (\cite{DG}, \cite{Ge}) that to every MV-algebra there can be assigned a so-called coupled semiring which bears all the information on that MV-algebra, i.~e., the latter can be recovered by its assigned coupled semiring. This fact inspired us to modify the concept of a coupled semiring in order to get a similar representation for commutative basic algebras (\cite{CL1}) or for general basic algebras (\cite{CL2}).

Every MV-algebra is indeed a residuated lattice satisfying the double negation law, the prelinearity and the divisibility condition (see \cite{Belo} for details). Hence we try to find a representation by means of some sort of coupled semirings also for the more general class of residuated lattices. In fact, we are successful in the case where the double negation law is assumed.

This shows that the construction of a coupled semiring from \cite{DG} and \cite{Ge} is quite general and it can be applied in the fairly general case of residuated lattices satisfying the double negation law. For similar categorical considerations see \cite{Bell}.

Finally, we want to stress the importance of semirings treated in the paper in applications and in the context of tropical geometry, see e.~g.\ \cite P.

We start with the definition of a residuated lattice.

\begin{Definition}\label{def1}
A {\em residuated lattice} is an algebra ${\mathcal L}=(L,\vee,\wedge,\otimes,\rightarrow,0,1)$ of type $(2,2,2,$ $2,0,0)$ satisfying the following axioms for all $x,y,z\in L$:
\begin{enumerate}
\item[{\rm(i)}] $(L,\vee,\wedge,0,1)$ is a bounded lattice.
\item[{\rm(ii)}] $(L,\otimes,1)$ is a commutative monoid.
\item[{\rm(iii)}] $x\leq y\rightarrow z$ if and only if $x\otimes y\leq z$
\end{enumerate}
\end{Definition}

\begin{Remark}\label{rem1}
Condition {\rm(iii)} is called the {\em adjointness property}.
\end{Remark}

As a source for elementary properties of residuated lattices see the monograph by \\
B\v elohl\'avek {\rm(\cite{Belo})}. We will work with residuated lattices having one more property.

\begin{Definition}\label{def2}
Let ${\mathcal L}=(L,\vee,\wedge,\otimes,\rightarrow,0,1)$ be a residuated lattice. On $L$ we define two further operations as follows:
\begin{eqnarray*}
\neg x & := & x\rightarrow0\mbox{ and} \\
x\oplus y & := & \neg(\neg x\otimes\neg y)
\end{eqnarray*}
for all $x,y\in L$. Further, we say that ${\mathcal L}$ satisfies the {\em double negation law} if $\neg\neg x=x$ for all $x\in L$.
\end{Definition}

If $(L,\vee,\wedge,\otimes,\rightarrow,0,1)$ is a residuated lattice satisfying the double negation law then $(L,\oplus,\neg,0)$ need not be an MV-algebra. This can be seen from the following example:

\begin{Example}\label{ex2}
{\rm(}cf.\ {\rm\cite{RS})} If $(L,\vee,\wedge,0,1)$ denotes the bounded lattice given by the following Hasse diagram:

\vspace*{-5mm}

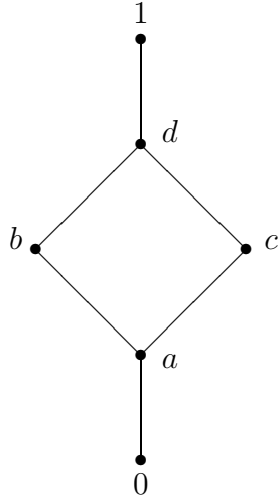
\begin{center}
\setlength{\unitlength}{7mm}
\begin{picture}(6,10)
\put(3,1){\circle*{.2}}
\put(3,3){\circle*{.2}}
\put(1,5){\circle*{.2}}
\put(5,5){\circle*{.2}}
\put(3,7){\circle*{.2}}
\put(3,9){\circle*{.2}}
\put(3,1){\line(0,1){2}}
\put(3,3){\line(-1,1){2}}
\put(3,3){\line(1,1){2}}
\put(3,7){\line(-1,-1){2}}
\put(3,7){\line(1,-1){2}}
\put(3,7){\line(0,1){2}}
\put(2.85,.35){$0$}
\put(3.4,2.75){$a$}
\put(.5,5){$b$}
\put(5.35,5){$c$}
\put(3.4,7.05){$d$}
\put(2.85,9.3){$1$}
\end{picture}
\end{center}

\vspace*{-5mm}

and we define binary operations $\otimes$ and $\rightarrow$ on $L$ as follows:
\[
\begin{array}{c|cccccc}
\otimes & 0 & a & b & c & d & 1 \\
\hline
0 & 0 & 0 & 0 & 0 & 0 & 0 \\
a & 0 & 0 & 0 & 0 & 0 & a \\
b & 0 & 0 & b & 0 & b & b \\
c & 0 & 0 & 0 & c & c & c \\
d & 0 & 0 & b & c & d & d \\
1 & 0 & a & b & c & d & 1
\end{array}
\quad\quad\quad
\begin{array}{c|cccccc}
\rightarrow & 0 & a & b & c & d & 1 \\
\hline
0 & 1 & 1 & 1 & 1 & 1 & 1 \\
a & d & 1 & 1 & 1 & 1 & 1 \\
b & c & c & 1 & c & 1 & 1 \\
c & b & b & b & 1 & 1 & 1 \\
d & a & a & b & c & 1 & 1 \\
1 & 0 & a & b & c & d & 1
\end{array}
\]
then we have
\[
\begin{array}{c|cccccc}
x & 0 & a & b & c & d & 1 \\
\hline
\neg x & 1 & d & c & b & a & 0
\end{array}
\quad\quad\quad
\begin{array}{c|cccccc}
\oplus & 0 & a & b & c & d & 1 \\
\hline
0 & 0 & a & b & c & d & 1 \\
a & a & a & b & c & 1 & 1 \\
b & b & b & b & c & 1 & 1 \\
c & c & c & 1 & c & 1 & 1 \\
d & d & 1 & 1 & 1 & 1 & 1 \\
1 & 1 & 1 & 1 & 1 & 1 & 1
\end{array}
\]
and $(L,\vee,\wedge,\otimes,\rightarrow,0,1)$ is a residuated lattice satisfying the double negation law which is neither prelinear nor divisible and hence not an {\rm MV}-algebra.
\end{Example}

The following properties of residuated lattices are well-known (cf.\ Theorems~2.17, 2.25, 2.27, 2.30 and 2.40 of \cite{Belo}).

\begin{Lemma}\label{lem1}
Let ${\mathcal L}=(L,\vee,\wedge,\otimes,\rightarrow,0,1)$ be a residuated lattice and $a,b,c\in L$. Then the following hold:
\begin{enumerate}
\item[{\rm(i)}] $a\leq b$ if and only if $a\rightarrow b=1$,
\item[{\rm(ii)}] $a\rightarrow1=1$,
\item[{\rm(iii)}] $a\otimes0=0$,
\item[{\rm(iv)}] $a\otimes(b\vee c)=(a\otimes b)\vee(a\otimes c)$,
\item[{\rm(v)}] $a\rightarrow b=((a\rightarrow b)\rightarrow b)\rightarrow b$,
\item[{\rm(vi)}] $\neg\neg\neg a=\neg a$,
\item[{\rm(vii)}] $\neg0=1$,
\item[{\rm(viii)}] $\neg1=0$ and
\item[{\rm(ix)}] $\neg(a\vee b)=\neg a\wedge\neg b$.
\end{enumerate}
If, moreover, ${\mathcal L}$ satisfies the double negation law then
\begin{enumerate}
\item[{\rm(x)}] $a\rightarrow b=\neg(a\otimes\neg b)$.
\end{enumerate}
\end{Lemma}

The following lemma is straightforward.

\begin{Lemma}\label{lem3}
Let ${\mathcal L}=(L,\vee,\wedge,\otimes,\rightarrow,0,1)$ be a residuated lattice satisfying the double negation law and $a,b,c\in L$. Then the following hold:
\begin{enumerate}
\item[{\rm(i)}] $a\otimes b=\neg(\neg a\oplus\neg b)$,
\item[{\rm(ii)}] $\neg(a\wedge b)=\neg a\vee\neg b$ and
\item[{\rm(iii)}] $a\rightarrow b=\neg a\oplus b$.
\end{enumerate}
\end{Lemma}

A further concept we need is that of a commutative semiring. Since within the literature there exist different definitions of this concept we present the definition taken from \cite{Go} or \cite{KS}.

\begin{Definition}\label{def3}
A {\em commutative semiring} is an algebra ${\mathcal S}=(S,+,\cdot,0,1)$ of type $(2,2,0,0)$ satisfying the following conditions for all $x,y,z\in S$:
\begin{enumerate}
\item[{\rm(i)}] $(S,+,0)$ and $(S,\cdot,1)$ are commutative monoids.
\item[{\rm(ii)}] $x(y+z)=xy+xz$
\item[{\rm(iii)}] $x0=0$
\end{enumerate}
\end{Definition}

In \cite{DG} and \cite{Ge} MV-algebras are represented by certain coupled semirings. In \cite{CL1} we used so-called coupled near semirings in order to represent commutative basic algebras. In \cite{CL2} we did the same job with coupled right near semirings for basic algebras. In order to represent residuated lattices we define the following notion:

\begin{Definition}\label{def5}
A {\em general coupled semiring} is an ordered triple
\[
((A,\vee,\cdot,0,1),(A,\wedge,\ast,1,0),\alpha)
\]
satisfying the following conditions for all $x,y\in A$:
\begin{enumerate}
\item[{\rm(i)}] $(A,\vee,\cdot,0,1)$ and $(A,\wedge,\ast,1,0)$ are commutative semirings.
\item[{\rm(ii)}] $(A,\vee,\wedge)$ is a lattice.
\item[{\rm(iii)}] $\alpha$ is an isomorphism from $(A,\vee,\cdot,0,1)$ to $(A,\wedge,\ast,1,0)$.
\item[{\rm(iv)}] $\alpha(\alpha(x))=x$
\item[{\rm(v)}] $x\leq y$ if and only if $\alpha(x)\ast y=1$
\end{enumerate}
\end{Definition}

We are now able to formulate and prove our first theorem.

\begin{Theorem}\label{th1}
Let ${\mathcal L}=(L,\vee,\wedge,\otimes,\rightarrow,0,1)$ be a residuated lattice satisfying the double negation law. Then
\[
{\bf C}({\mathcal L}):=((L,\vee,\otimes,0,1),(L,\wedge,\oplus,1,0),\neg)
\]
is a general coupled semiring.
\end{Theorem}

\begin{proof}
Let $a,b\in L$. According to Definition~\ref{def1}, $(L,\vee,\wedge)$ is a lattice. Because of Definition~\ref{def1} and Lemma~\ref{lem1}, $(L,\vee,\otimes,0,1)$ is a commutative semiring. According to Lemmata~\ref{lem1} and \ref{lem3}, $\neg$ is an involutory isomorphism from $(L,\vee,\otimes,0,1)$ to $(L,\wedge,\oplus,1,0)$ and the latter therefore a commutative semiring, too. Finally, because of Lemmata~\ref{lem1} and \ref{lem3} the following are equivalent:
\begin{eqnarray*}
a & \leq & b \\
a\rightarrow b & = & 1 \\
\neg a\oplus b & = & 1
\end{eqnarray*}
Summing up, ${\bf C}({\mathcal L})$ is a general coupled semiring.
\end{proof}

\begin{Remark}\label{rem3}
Hence, $(L,\wedge,\oplus,1,0)$ is the so-called min-plus semiring, thus if $L=\mathbb R\cup\{\infty\}$ then $(L,\wedge,\oplus)$ is nothing else than a tropical semiring, see e.~g.\ {\rm\cite P}.
\end{Remark}

If the residuated lattice ${\mathcal L}=(L,\vee,\wedge,\otimes,\rightarrow,0,1)$ is an MV-algebra then the general coupled semiring ${\bf C}({\mathcal L})$ coincides with that introduced in \cite{DG} and \cite{Ge}.

Now we are going to prove the converse.

\begin{Theorem}\label{th2}
Let ${\mathcal C}=((A,\vee,\cdot,0,1),(A,\wedge,\ast,1,0),\alpha)$ be a general coupled semiring and define
\[
x\rightarrow y:=\alpha(x)\ast y
\]
for all $x,y\in A$. Then
\[
{\bf L}({\mathcal C}):=(A,\vee,\wedge,\cdot,\rightarrow,0,1)
\]
is a residuated lattice satisfying the double negation law.
\end{Theorem}

\begin{proof}
Let $a,b,c\in A$. According to Definition~\ref{def5}, $(A,\vee,\wedge)$ is a lattice. Since $0$ is the neutral element with respect to $\vee$, it is the least element of this lattice. Analogously, it follows that $1$ is the greatest element of this lattice. Moreover, $(A,\cdot,1)$ is a commutative monoid. In order to prove that ${\bf L}({\mathcal C})$ is a residuated lattice we have to check the adjointness property. Now according to Definition~\ref{def5} the following statements are equivalent:
\begin{eqnarray*}
ab & \leq & c \\
\alpha(ab)\ast c & = & 1 \\
(\alpha(a)\ast\alpha(b))\ast c & = & 1 \\
\alpha(a)\ast(\alpha(b)\ast c) & = & 1 \\
a & \leq & \alpha(b)\ast c \\
a & \leq & b\rightarrow c.
\end{eqnarray*}
Thus ${\bf L}({\mathcal C})$ is a residuated lattice. It remains to check the double negation law. Now we have
\[
\neg a=a\rightarrow0=\alpha(a)\ast0=\alpha(a)
\]
and hence
\[
\neg\neg a=\alpha(\alpha(a))=a.
\]
\end{proof}

Finally, we prove that the above correspondence between residuated lattices satisfying the double negation law and general coupled semirings is one-to-one.

\begin{Theorem}\label{th3}
Let ${\mathcal L}=(L,\vee,\wedge,\otimes,\rightarrow,0,1)$ be a residuated lattice satisfying the double negation law. Then ${\bf L}({\bf C}({\mathcal L}))={\mathcal L}$.
\end{Theorem}

\begin{proof}
If ${\bf C}({\mathcal L})=((L,\vee,\otimes,0,1),(L,\wedge,\oplus,1,0),\neg)$ and ${\bf L}({\bf C}({\mathcal L}))=(L,\vee,\wedge,\otimes,\Rightarrow,0,1)$ and $a,b\in L$ then $a\Rightarrow b=\neg a\oplus b=a\rightarrow b$ according to Lemma~\ref{lem3}.
\end{proof}

\begin{Theorem}\label{th4}
Let ${\mathcal C}=((A,\vee,\cdot,0,1),(A,\wedge,\ast,1,0),\alpha)$ be a general coupled semiring. Then ${\bf C}({\bf L}({\mathcal C}))={\mathcal C}$.
\end{Theorem}

\begin{proof}
If ${\bf L}({\mathcal C})=(A,\vee,\wedge,\cdot,\rightarrow,0,1)$, ${\bf C}({\bf L}({\mathcal C}))=((A,\vee,\cdot,0,1),(A,\wedge,\oplus,1,$ $0),\neg)$ and $a,b\in A$ then
\begin{eqnarray*}
\neg a & = & a\rightarrow0=\alpha(a)\ast0=\alpha(a)\mbox{ and} \\
a\oplus b & = & \neg(\neg a\cdot\neg b)=\alpha(\alpha(a)\alpha(b))=\alpha(\alpha(a))\ast\alpha(\alpha(b))=a\ast b.
\end{eqnarray*}
\end{proof}

In what follows, we are going to extend our investigation concerning the mutual relationship between residuated lattices satisfying the double negation law and semirings to the general case where no double negation law is assumed. As before, we will denote by $\leq$ the induced order of a residuated lattice.

\begin{Definition}\label{def4}
A {\em tied semiring} is an ordered triple
\[
((A,\vee,\cdot,0,1),(B,\wedge,\ast,1,0),\alpha)
\]
satisfying the following conditions for all $x,y\in B$:
\begin{enumerate}
\item[{\rm(i)}] $(A,\vee,\cdot,0,1)$ and $(B,\wedge,\ast,1,0)$ are commutative semirings and $B\subseteq A$.
\item[{\rm(ii)}] $(A,\vee,\wedge)$ a lattice.
\item[{\rm(iii)}] $\alpha$ is a homomorphism from $(A,\vee,\cdot,0,1)$ onto $(B,\wedge,\ast,1,0)$.
\item[{\rm(iv)}] $\alpha|B$ is a homomorphism from $(B,\wedge,\ast,1,0)$ to $(A,\vee,\cdot,0,1)$.
\item[{\rm(v)}] $\alpha(\alpha(x))=x$
\item[{\rm(vi)}] $x\leq y$ if and only if $\alpha(x)\ast y=1$
\end{enumerate}
\end{Definition}

\begin{Remark}\label{rem2}
Let ${\mathcal L}=(L,\vee,\wedge,\otimes,\rightarrow,0,1)$ be a residuated lattice. Put
\[
\neg L:=\{\neg x\,|\,x\in L\}.
\]
Because of {\rm(vi)} of Lemma~\ref{lem1}, $\neg L=\{x\in L\,|\,\neg\neg x=x\}$. It is well-known that $\neg L$ need not be a subuniverse of ${\mathcal L}$ if the double negation law is not assumed. It should be noted that there exists a subuniverse $A$ of ${\mathcal L}$ with the property that $\neg A$ is a subuniverse of ${\mathcal L}$, too, namely $A=\{0,1\}$.
\end{Remark}

The connection between residuated lattices not necessarily satisfying the double negation law and tied semirings is as follows:

\begin{Theorem}\label{th6}
Let ${\mathcal L}=(L,\vee,\wedge,\otimes,\rightarrow,0,1)$ be a residuated lattice and assume $A$ to be a subuniverse of ${\mathcal L}$ such that $\neg A:=\{\neg x\,|\,x\in A\}$ is a subuniverse of ${\mathcal L}$, too. Moreover, assume that the following condition holds:
\begin{enumerate}
\item[{\rm(i)}] $\neg(x\otimes y)=\neg x\oplus\neg y$ for all $x,y\in A$.
\end{enumerate}
Then
\[
{\bf Y}({\mathcal L},A):=((A,\vee,\otimes,0,1),(\neg A,\wedge,\oplus,1,0),\neg)
\]
is a tied semiring.
\end{Theorem}

\begin{proof}
Let $a,b\in A$ and $c,d\in\neg A$. Since $A$ is a subuniverse of ${\mathcal L}$, $(A,\vee,\otimes,0,1)$ is a commutative semiring according to Definition~\ref{def1} and Lemma~\ref{lem1}. Since $\neg A$ and $A$ are subuniverses of ${\mathcal L}$ and (i) holds, we have that $\neg A$ is a subuniverse of $(L,\wedge,\oplus,1,0)$, too. Moreover,
\[
\neg A=\{x\rightarrow0\,|\,x\in A\}\subseteq A.
\]
Since $A$ is a subuniverse of ${\mathcal L}$, $(A,\vee,\wedge)$ is a lattice. Now, according to Lemma~\ref{lem1} and (i) we have
\begin{eqnarray*}
\neg(a\vee b) & = & \neg a\wedge\neg b\mbox{ and} \\
\neg(a\otimes b) & = & \neg a\oplus\neg b.
\end{eqnarray*}
Hence, $\neg$ is a homomorphism from $(A,\vee,\otimes,0,1)$ onto $(\neg A,\wedge,\oplus,1,0)$ and therefore the latter is a commutative semiring, too. Since $\neg A$ is a subuniverse of ${\mathcal L}$, $(\neg A,\vee,\wedge,\otimes,\rightarrow,0,1)$ is a residuated lattice satisfying the double negation law. Hence because of $\neg\neg A\subseteq A$ and Lemma~\ref{lem3}, $\neg|(\neg A)$ is a homomorphism from $(\neg A,\wedge,\oplus,1,0)$ to $(A,\vee,\otimes,0,1)$. Since $c\in\neg A$ we have that $\neg\neg c=c$ because of Lemma~\ref{lem1}. Finally, since $(\neg A,\vee,\wedge,\otimes,\rightarrow,0,1)$ is a residuated lattice satisfying the double negation law, the following are equivalent:
\begin{eqnarray*}
c & \leq & d \\
c\rightarrow d & = & 1 \\
\neg c\oplus d & = & 1
\end{eqnarray*}
Summing up, ${\bf Y}({\mathcal L},A)$ is a tied semiring.
\end{proof}

\begin{Example}\label{ex3}
Let $(A,\leq)$ with $A=\{0,a,1\}$ and $0<a<1$ be a three-element chain and define a binary operation $\rightarrow$ on $A$ by
\[
\begin{array}{c|ccc}
\rightarrow & 0 & a & 1 \\
\hline
0 & 1 & 1 & 1 \\
a & 0 & 1 & 1 \\
1 & 0 & a & 1
\end{array}
\]
It is easy to see that ${\mathcal L}=(A,\vee,\wedge,\wedge,\rightarrow,0,1)$ is a residuated lattice with
\[
\begin{array}{c|ccc}
x & 0 & a & 1 \\
\hline
\neg x & 1 & 0 & 0
\end{array}
\quad\quad\quad
\begin{array}{c|ccc}
\oplus & 0 & a & 1 \\
\hline
0 & 0 & 1 & 1 \\
a & 1 & 1 & 1 \\
1 & 1 & 1 & 1
\end{array}
\]
Hence, ${\mathcal L}$ does not satisfy the double negation law. Put $B:=\neg A=\{0,1\}$. Then $B$ is a subuniverse of ${\mathcal L}$ and $\oplus|B=\vee|B$. Moreover, {\rm(i)} of Theorem~\ref{th6} holds. According to Theorem~\ref{th6},
\[
{\bf Y}({\mathcal L},A)=((A,\vee,\wedge,0,1),(B,\wedge,\vee,1,0),\neg)
\]
is a tied semiring.
\end{Example}

\begin{Example}\label{ex1}
Let ${\mathcal L}_1=(L,\oplus,\neg,0)$ denote the basic algebra with $L=\{0,a,\neg a,b,\neg b,1\}$ and
\[
\begin{array}{r|rrrrrr}
\oplus &      0 &      a & \neg a &      b & \neg b & 1 \\
\hline
     0 &      0 &      a & \neg a &      b & \neg b & 1 \\
		 a &      a &      a &      1 & \neg b & \neg b & 1 \\
\neg a & \neg a &      1 & \neg a & \neg a &      1 & 1 \\
		 b &      b & \neg b & \neg a & \neg a &      1 & 1 \\
\neg b & \neg b & \neg b &      1 &      1 &      1 & 1 \\
     1 &      1 &      1 &      1 &      1 &      1 & 1
\end{array}
\quad\quad\quad
\begin{array}{r|rrrrrr}
     x & 0 &      a & \neg a &      b & \neg b & 1 \\
\hline
\neg x & 1 & \neg a &      a & \neg b &      b & 0
\end{array}
\]
According to {\rm\cite{BH}} every finite basic algebra, hence also ${\mathcal L}_1$, can be considered as an MV-algebra ${\mathcal L}_2=(L,\vee,\wedge,\otimes,\rightarrow,0,1)$ and hence also as a residuated lattice satisfying the double negation law. The operations of ${\mathcal L}_2$ are as follows:
\[
\begin{array}{r|rrrrrr}
  \vee &      0 &      a & \neg a &      b & \neg b & 1 \\
\hline
     0 &      0 &      a & \neg a &      b & \neg b & 1 \\
	   a &      a &      a &      1 & \neg b & \neg b & 1 \\
\neg a & \neg a &      1 & \neg a & \neg a &      1 & 1 \\
     b &      b & \neg b & \neg a &      b & \neg b & 1 \\
\neg b & \neg b & \neg b &      1 & \neg b & \neg b & 1 \\
     1 &      1 &      1 &      1 &      1 &      1 & 1
\end{array}
\quad\quad\quad
\begin{array}{r|rrrrrr}
\wedge & 0 & a & \neg a & b & \neg b & 1 \\
\hline
     0 & 0 & 0 &      0 & 0 &      0 &      0 \\
	   a & 0 & a &      0 & 0 &      a &      a \\
\neg a & 0 & 0 & \neg a & b &      b & \neg a \\
     b & 0 & 0 &      b & b &      b &      b \\
\neg b & 0 & a &      b & b & \neg b & \neg b \\
     1 & 0 & a & \neg a & b & \neg b &      1
\end{array}
\]
\[
\begin{array}{r|rrrrrr}
\otimes & 0 & a & \neg a & b & \neg b & 1 \\
\hline
      0 & 0 & 0 &      0 & 0 &      0 &      0 \\
		  a & 0 & a &      0 & 0 &      a &      a \\
 \neg a & 0 & 0 & \neg a & b &      b & \neg a \\
	    b & 0 & 0 &      b & 0 &      0 &      b \\
 \neg b & 0 & a &      b & 0 &      a & \neg b \\
      1 & 0 & a & \neg a & b & \neg b &      1
\end{array}
\quad\quad\quad
\begin{array}{r|rrrrrr}
\rightarrow &      0 &      a & \neg a &      b & \neg b & 1 \\
\hline
          0 &      1 &      1 &      1 &      1 &      1 & 1 \\
		      a & \neg a &      1 & \neg a & \neg a &      1 & 1 \\
     \neg a &      a &      a &      1 & \neg b & \neg b & 1 \\
	        b & \neg b & \neg b &      1 &      1 &      1 & 1 \\
     \neg b &      b & \neg b & \neg a & \neg a &      1 & 1 \\
          1 &      0 &      a & \neg a &      b & \neg b & 1
\end{array}
\]
Put $A:=\{0,a,\neg a,1\}$. Then $A$ is a subuniverse of ${\mathcal L}_2$ and so is $\neg A=A$. Moreover, {\rm(i)} of Theorem~\ref{th6} holds. Hence
\[
{\bf Y}({\mathcal L}_2,A):=((A,\vee,\otimes,0,1),(\neg A,\wedge,\oplus,1,0),\neg)
\]
is a non-trivial tied semiring.
\end{Example}

Now we can prove a counterpart of the last theorem.

\begin{Theorem}\label{th5}
Let ${\mathcal Y}=((A,\vee,\cdot,0,1),(B,\wedge,\ast,1,0),\alpha)$ be a tied semiring and define
\[
x\rightarrow y:=\alpha(x)\ast y
\]
for all $x,y\in B$. Then
\[
{\bf A}({\mathcal Y}):=(B,\vee,\wedge,\cdot,\rightarrow,0,1)
\]
is a residuated lattice satisfying the double negation law and $\neg x=\alpha(x)$ for all $x\in B$.
\end{Theorem}

\begin{proof}
Let $a,b,c\in B$. Then
\begin{eqnarray*}
a\vee b & = & \alpha(\alpha(a))\vee\alpha(\alpha(b))=\alpha(\alpha(a)\wedge\alpha(b))\in B\mbox{ and} \\
ab & = & \alpha(\alpha(a))\alpha(\alpha(b))=\alpha(\alpha(a)\ast\alpha(b))\in B.
\end{eqnarray*}
According to Definition~\ref{def5}, $(A,\vee,\wedge)$ is a lattice. Since $0$ is the neutral element with respect to $\vee$, it is the least element of this lattice. Analogously, it follows that $1$ is the greatest element of this lattice. Hence $(A,\vee,\wedge,0,1)$ is a bounded lattice with subuniverse $B$. This shows that $(B,\vee,\wedge,0,1)$ is a bounded lattice, too. Since $(A,\cdot,1)$ is a commutative monoid with subuniverse $B$ we have that $(B,\cdot,1)$ is a commutative monoid, too. Moreover, the following are equivalent:
\begin{eqnarray*}
ab & \leq & c \\
\alpha(ab)\ast c & = & 1 \\
(\alpha(a)\ast\alpha(b))\ast c & = & 1 \\
\alpha(a)\ast(\alpha(b)\ast c) & = & 1 \\
a & \leq & \alpha(b)\ast c \\
a & \leq & b\rightarrow c.
\end{eqnarray*}
Finally,
\[
\neg a=a\rightarrow0=\alpha(a)\ast0=\alpha(a)
\]
and hence
\[
\neg\neg a=\alpha(\alpha(a))=a.
\]
\end{proof}

Authors' addresses:

Ivan Chajda \\
Palack\'y University Olomouc \\
Faculty of Science \\
Department of Algebra and Geometry \\
17.\ listopadu 12 \\
77146 Olomouc \\
Czech Republic \\
ivan.chajda@upol.cz

Helmut L\"anger \\
TU Wien \\
Faculty of Mathematics and Geoinformation \\
Institute of Discrete Mathematics and Geometry \\
Wiedner Hauptstra\ss e 8-10 \\
1040 Vienna \\
Austria \\
helmut.laenger@tuwien.ac.at
\end{document}